\documentclass[11pt]{amsart}
\usepackage{amssymb,amsmath}

\makeatletter
\theoremstyle{plain}
 \newtheorem{thm}{Theorem}[section]
\newtheorem{thm*}{Theorem}
 \newtheorem{lem}[thm]{Lemma}
 \newtheorem{prop}[thm]{Proposition}
 \newtheorem{cor}[thm]{Corollary}
 \numberwithin{equation}{section} 
\numberwithin{figure}{section} 
 \theoremstyle{plain}
 \theoremstyle{definition}
 
 \newtheorem{rem}[thm]{Remark}

\newcommand{\cA}{{{\mathcal A}}}

\newcommand{\cM}{{{\mathcal M}}}

\newcommand{\C}{{{\mathbb C}}}

\newcommand{\bO}{{{\mathbb O}}}
\newcommand{\K}{{{\mathbb K}}}
\newcommand{\R}{{{\mathbb R}}}

\newcommand{\bH}{{{\bf H}}}
\newcommand{\mH}{{{\mathbb H}}}

\newcommand{\cC}{{{\mathcal C}}}
\newcommand{\fp}{{{\mathfrak p}}}

\newcommand{\rF}{{{\rm F}}}

\newcommand{\X}{{{\mathbb X}}}

\newcommand{\fH}{{{\mathfrak H}}}

\makeatother

\begin{document}

\title[M\"obius  rigidity of invariant metrics]{M\"obius  rigidity of invariant metrics in boundaries of symmetric spaces of rank-1}

\author[I.D. Platis \& V. Schroeder]{I.D. Platis \& V. Schroeder}

\email{jplatis@math.uoc.gr, viktor.schroeder@math.uzh.ch}
\address{Department of Mathematics and Applied Mathematics\\ 
University of Crete\\
University Campus\\
GR 700 13 Voutes\\ Heraklion Crete\\Greece}

\bigskip

\address{Institut f\"ur Mathematik\\
Universit\"at Z\"urich\\
Winterthurerstrasse 190\\
CH 8057 Z\"urich\\Switzerland}

\begin{abstract}
Let $\bH^n_\K$ denote the symmetric space of rank-1 and of non-compact type and let $d_\fH$ be the Kor\'anyi metric defined on its boundary. We prove that if $d$ is a metric on $\partial\bH^n_\K$ such that all Heisenberg similarities are $d$-M\"obius maps, then under a topological condition $d$ is a constant multiple of a power of  $d_\fH$. 

\end{abstract}

\date{{17 March 2016}
}

\maketitle

\section{Introduction and Statement of Results}
Let $(S, d)$ be a metric space and suppose that there exists a remote point $\infty$ such that $\overline{S}=S\cup\{\infty\}$ is compact. We may extend $d$ to the compactification by agreeing  that $d(p,\infty)=+\infty$ for every $p\in S$ and also  $d(\infty,\infty)=0$. A natural metric cross-ratio $|\X^d|$ is defined for each quadruple of pairwise distinct points $\fp=(p_1,p_2,p_3,p_4)$ by setting
$$
|\X^d|(\fp)=\frac{d(p_4,p_2)}{d(p_4,p_1)}\cdot\frac{d(p_3,p_1)}{d(p_3,p_2)},
$$
with obvious modifications when one of the points is $\infty$. The {\it M\"obius group} $\cM_d=\cM_d(S)$ is the group of homeomorphisms of $\overline{S}$ which leave $|\X^d|$ invariant for each quadruple $\fp$ of pairwise distinct points of $\overline{S}$. The {\it similarity group} (or {\it homothety group}) ${\rm Sim}_d={\rm Sim}_d(S)$ is the subgroup of $\cM_d$ comprising homeomorphisms $\phi$ such that there exists a positive constant  $K(\phi)$ which satisfies $d\left(\phi(p),\phi(q)\right)=K(\phi)\cdot d(p,q),$ for all $p,q\in S$. If $K(\phi)=1$ then $\phi$ is an {\it isometry}. 

Given two metrics $d_1$ and $d_2$ on $S$ for which $\infty$ is their remote point, we say that they define the same {\it M\"obius structure} on $S$ if $|\X^{d_1}|=|\X^{d_2}|$. This happens if and only if $d_1$ and $d_2$ are homothetic, that is, there exists a $c>0$ such that $d_1=c\cdot d_2$ (see Lemma 2.1 of \cite{BS}). It is immediate that for such metrics we have $\cM_{d_1}=\cM_{d_2}$. 

In this paper we are concerned with the following problem: Suppose that we are given a representative $d_1$ of a M\"obius structure on $\overline{S}$ and suppose also that there exists a metric $d_2$ defined on $S$ which has the same remote point as $d_1$ and also satisfies $\cM_{d_1}\subseteq \cM_{d_2}$. Then, assuming that $\cM_{d_1}$ is sufficiently rich, what can we say about the relation of $d_1$ and $d_2$? 

We give a quite precise answer to this question in the case where $\overline{S}$ is the boundary of a symmetric space of rank-1 and of non-compact type and the given M\"obius structure is the canonical one, i.e., the one arising from the Kor\'anyi metric.  We recall below some known facts about the aforementioned notions; for details we refer to Section \ref{sec:prel}.
Let $\K=\R,\C,\mH$ or $\bO$ be the set of real, complex, quaternionic and octonionic numbers, respectively. With $\bH^n_\K$ we shall thereafter denote the $n$-dimensional $\K$-hyperbolic space (in the case $\K=\bO$, $n=2$). Symmetric spaces of rank-1 and of non-compact type are necessarily the $\K$-hyperbolic spaces. There is a natural conjugation $z\mapsto\overline{z}$ in $\K$ (in the case when $\K=\R$ this is of course the identity map); denote by $\Im(\K)$ the subset of $\K$ comprising points $z$ such that $z=-\overline{z}$. The boundary $\partial\bH^n_\K$ is a sphere, isomorphic to the one point compactification of the $\K$-{\it Heisenberg group} $\fH_\K$: That is, the set $\K^{n-1}\times\Im(\K)$ endowed with the group multiplication
$$
(\zeta,v)*(\zeta',v')=\left(\zeta+\zeta',v+v+2\omega(\zeta,\zeta')\right),
$$
for each $(\zeta,v),(\zeta',v')\in\fH_\K$. Here, $\omega$ is the usual symplectic form in $\K^{n-1}$. A natural conjugation $J$ is defined on $\K^{n-1}\times\Im(\K)$ as the restriction of the conjugation of $\bH^n_\K$; that is,  $(\zeta,v)\mapsto (\overline{\zeta},-v)$. There is a gauge $|\cdot|_\K$ on $\fH_\K$, called the Kor\'anyi gauge, which is given for each $(\zeta,v)$ by
$$
|(\zeta,v)|_\K=|-\|\zeta\|^2+v|^{1/2}=\left(\|\zeta\|^4+|v|^2\right)^{1/4}.
$$
The Kor\'anyi (or, Kor\'anyi-Cygan) metric $d_\fH$ is then defined in $\fH_\K$ by
\begin{equation}\label{eq:Kor}
d_\fH\left((\zeta,v)\;,\;(\zeta',v')\right)=\left|(\zeta',v')^{-1}*(\zeta,v)\right|_\K,
\end{equation}
for each $(\zeta,v),(\zeta',v')\in\fH_\K$. 
The canonical M\"obius structure of $\partial\bH^n_\K$ is the one arising from $d_\fH$. The M\"obius group $\cM_{d_\fH}$ of the metric $d_\fH$ is generated by the similarities of $d_\fH$ and an inversion $I$ defined on $\partial\bH^n_\K$, see Section \ref{sec:prel} for details.

One of our main results can be stated in the following way:

\medskip

\noindent{\bf Theorem.} {\it Let $d$ be a metric on $\partial\bH^n_\K$, which induces the same topology and has the same remote point as $d_\fH$. Assume further that $\cM_d=\cM_{d_\fH}$. Then there exist $\alpha\in(0,1]$ and $\beta>0$ such that $d=\beta\cdot d^\alpha_\fH$.}

\medskip

This special result is embedded in much more general statements: Suppose that we  are given a metric $d$ in $\partial\bH^n_\K$  such that:  

\medskip

\begin{enumerate}
 \item [{{\bf (Sim)}}] ${\rm Sim}_{d_\fH}\subseteq{\rm Sim}_d$.
\end{enumerate}

\medskip

We shall further presuppose the next quite plausible topological condition:

\medskip

\begin{enumerate}
 \item  [{{\bf (Top)}}] 
 If a sequence of points $p_\nu\in\partial\bH^n_\K$, $\nu\in\mathbb{N}$, is converging in the $d_\fH$-topology of $\partial\bH^n_\K$, then it converges in the $d$-topology of $\partial\bH^n_\K$.
\end{enumerate}

\medskip

Before stating our results, we  list the following set of conditions:

\medskip

\begin{enumerate}
  \item  [{{\bf (Inv)}}] Inversion $I$ is  in $\cM_d$.
  \item [{{\bf ($\alpha$-H\"ol)}}] There exists a constant $\alpha\in (0,1]$ such that the identity map of $(\partial\bH^n_\K, d_\fH)\to(\partial\bH^n_\K, d)$ is H\"older continuous with exponent $\alpha$. Explicitly, there exists a $\beta>0$ such that for each $p,q\in\partial\bH^n_\K$, 
  $$
  d(p,q)\le\beta\cdot d_\fH^\alpha(p,q).
  $$
  \item [{{\bf (G)}}] There exists an $\alpha\in (0,1]$ such that for each $p=(\zeta,v)\in\fH_\K$, $\K\neq\R$,
  \begin{equation*}
   d^{4/\alpha}(o,p)=d^{4/\alpha}\left(o,\Pi_{\K^{n-1}}(p)\right)+d^{4/\alpha}\left(\Pi_{\K^{n-1}}(p),\Pi_{\Im(\K)}(p)\right).
  \end{equation*}
Here, $o=({\bf 0}_{\K^{n-1}},{\bf 0}_{\Im(\K)})$ is the origin of $\fH_\K$, where ${\bf 0}_{\K^{n-1}}$ and ${\bf 0}_{\Im(\K)}$ are the origins of $\K^{n-1}$ and $\Im(\K)$, respectively. Moreover,  $\Pi_{\K^{n-1}}$,  $\Pi_{\Im(\K)}$ are defined respectively by $\Pi_{\K^{n-1}}(p)=\zeta$ and $\Pi_{\Im(\K)}(p)=v$.
\item [{{\bf (Eq)}}] The metric $d$ satisfies $$d\left(o,(1,0)\right)=d\left(o,(0,1)\right).$$ Here, $(1,0)=({\bf e}_1,{\bf 0}_{\Im(\K)})$, where ${\bf e}_1=(1,0,\dots,0)$ is the first coordinate vector of $\K^{n-1}$ and $(0,1)=({\bf 0}_{\K^{n-1}},{\bf f}_1)$ where ${\bf f}_1$ is the first coordinate vector of $\Im(\K)$.
\item [{{\bf (biLip)}}] The metric $d$ is bi-Lipschitz equivalent to $d_\fH^\alpha$ for some $\alpha\in (0,1]$.
\item [{{\bf ($\alpha$-Met)}}]  There exist  constants $\beta>0$ and  $\alpha\in(0,1]$ such that $d=\beta\cdot d^\alpha_\fH$.
\end{enumerate}

\medskip

Our first theorem assumes the least number of conditions for $d$, that is, {\bf (Sim)}  and {\bf (Top)}.

\begin{thm}\label{thm:main1}
Let $d$ be a metric in $\fH_\K$ and denote again by $d$ its extension to $\partial\bH^n_\K=\fH\cup\{\infty\}$. 
Then the following hold:
\begin{enumerate}
\item If $\K=\R$, conditions {\bf (Sim)} and {\bf (Top)} together imply {\bf ($\alpha$-Met)}.
\item If $\K\neq\R$, conditions {\bf (Sim)} and {\bf (Top)}  together imply {\bf ($\alpha$-H\"ol)}.
\end{enumerate} 
\end{thm}

We note the significant difference between the real case and all the other cases. With the least possible assumptions, the metric $d$ is a power of the Euclidean metric when $\K=\R$. The picture is entirely different in case when  $\K\neq\R$; metrics which satisfy {\bf ($\alpha$-H\"ol)} may be of nature entirely different from the one of a power of $d_\fH$ (e.g., the Carnot-Carath\'eodory metric). It is therefore quite necessary to add more assumptions for $d$ in this case. To that end, the strongest version of our first main result for the case $\K\neq\R$ follows:

\begin{thm}\label{thm:main}
With the assumptions of Theorem \ref{thm:main1}, suppose $\K\neq\R$. Then conditions {\bf (Sim)}, {\bf (Top)} and {\bf (Inv)} together imply {\bf ($\alpha$-Met)}. 
\end{thm}

But again, adding {\bf (Inv)} to our basic assumptions seems to be excessive for the proof of {\bf ($\alpha$-Met)}; in the case $\K=\R$ {\bf (Inv)} holds {\it a posteriori}. It is natural therefore to ask if {\bf (Inv)} is really necessary. Theorem \ref{thm:main1} tells us that we obtain a rather weak result if we entirely drop {\bf (Inv)}. It turns out though that {\bf ($\alpha$-Met)} follows by replacing {\bf (Inv)} with {\bf (G)} and {\bf (Eq)}. In fact, we have:

\begin{thm}\label{thm:mainw2}
 With the assumptions of Theorem \ref{thm:main1}, suppose $\K\neq\R$. Then:
\begin{enumerate}
 \item Conditions {\bf (Sim)}, {\bf (Top)} and {\bf (G)} together imply {\bf (biLip)}.
 \item Conditions {\bf (Sim)}, {\bf (Top)}, {\bf (G)} and {\bf (Eq)} together imply {\bf ($\alpha$-Met)}.
\end{enumerate} 
\end{thm}

Therefore, {\bf (Inv)} follows as a side result of  Theorem \ref{thm:mainw2} and we may further observe  that {\bf (G)} and {\bf (Eq)} hold vacuously in the real case. Moreover, {\bf (G)} can be replaced with an equivalent statement which is the closest to parallelogram law in the $\K$-Heisenberg group setting, $\K\neq\R$.  For this set for each $p\in\fH$,
$$
|p|=d(o,p).
$$
It turns out that if {\bf (Sim)} holds, then {\bf (G)} is equivalent to the following condition, see Proposition \ref{prop:p-l}:

\medskip

\begin{enumerate}
\item [{{\bf (P-L)}}] For $\alpha\in (0,1]$ and for each $p,q\in\fH$,
\begin{eqnarray*}
\left|p*q\right|^{4/\alpha}+ \left|p^{-1}*q\right|^{4/\alpha}+\left|p*q^{-1}\right|^{4/\alpha}+\left|p^{-1}*q^{-1}\right|^{4/\alpha}&=&\\
2\left(\left|\Pi_{\K^{n-1}}(p*q)\right|^{4/\alpha}+\left|\Pi_{\K^{n-1}}\left(p^{-1}*q\right)\right|^{4/\alpha}\right)&+&\\
\left|\Pi_{\Im(\K)}(p*q)\right|^{4/\alpha}+ \left|\Pi_{\Im(\K)}\left(p^{-1}*q\right)\right|^{4/\alpha}+\left|\Pi_{\Im(\K)}\left(p*q^{-1}\right)\right|^{4/\alpha}+\left|\Pi_{\Im(\K)}\left(p^{-1}*q^{-1}\right)\right|^{4/\alpha},
\end{eqnarray*}
where $\Pi_{\K^{n-1}}$ and $\Pi_{\Im(\K)}$ are the projections of $\fH$ to $\K^{n-1}$ and $\Im(\K)$, respectively. 
\end{enumerate} 

\medskip

Thus an equivalent to the second statement of Theorem \ref{thm:mainw2} (2) is:

\medskip

\begin{enumerate}
\item [{($2^{'}$)}] Conditions {\bf (Sim)}, {\bf (Top)}, {\bf (P-L)} and {\bf (Eq)} together imply {\bf ($\alpha$-Met)}.
\end{enumerate}

\medskip

Recall now that a metric $d$ defined in a space $S$ is Ptolemaean if for each quadruple of points $\fp=(p,q,r,s)$ of $S$ the following inequality is satisfied for all possible permutations of points in $\fp$:
\begin{equation}\label{eq:Ptolineq}
 d(p,r)\cdot d(q,s)\le d(p,q)\cdot d(r,s)+d(p,s)\cdot d(r,q).
\end{equation}
A Ptolemaean circle is a subset $\sigma$ of $S$ which is topologically equivalent to the unit circle $S^1$ and for each quadruple of points $\fp=(p,q,r,s)$ of $\sigma$ such that $p$ and $r$ separate $q$ and $s$, Inequality \ref{eq:Ptolineq} holds as an equality.
It is well known that $(\partial\bH^n_\K, d_\fH)$ is Ptolemaean, see for instance \cite{P}; therefore it is natural to ask which of the metrics $d$ who satisfy {\bf ($\alpha$-Met)} satisfies also:

\medskip

\begin{enumerate}
  \item [{{\bf (Ptol)}}] The metric $d$ is Ptolemaean.
  \end{enumerate}
  
  \medskip
  
  It turns out that we will also need:

  \medskip
  
  \begin{enumerate}
  \item [{{\bf (Circ)}}] The metric $d$ has a Ptolemaean circle.
 \end{enumerate}
 
 \medskip

 We have the following corollary to Theorem \ref{thm:mainw2}:

\begin{thm}\label{thm:pt}
Condition {\bf ($\alpha$-Met)} implies {\bf (Ptol)} and {\bf (Ptol)} together with {\bf (Circ)}  are equivalent to  {\bf ($1$-Met)}. Therefore a metric $d$ in $\partial\bH^n_\K$ which satisfies conditions: 
\begin{enumerate}
\item {\bf (Sim)}, {\bf (Top)} and {\bf (Circ)} if $\K=\R$ and
\item {\bf (Sim)}, {\bf (Top)}, {\bf (P-L)} or {\bf (G)}, {\bf (Eq)}  and {\bf (Circ)} if $\K\neq\R$,
\end{enumerate}
is necessarily a constant multiple of the Kor\'anyi metric $d_\fH$.
\end{thm}

This is in the spirit of the old result of Schoenberg, see \cite{S}. That particular result was on metrics which were derived from semi-norms which  also share  Ptolemaean property.

\thanks{ Part of this work was carried out while IDP was visiting University of Z\"urich, Switzerland. Hospitality of Instit\"ut f\"ur Mathematik, University of Z\"urich, is gratefully appreciated. 

\section{Preliminaries}\label{sec:prel}

This section is divided in two parts. In the first part (Section \ref{sec:hyp}), we state for clarity some known results about the Kor\'anyi metric in $\partial\bH^n_\K$ and its properties. In the second part (Section \ref{sec:Mob}), we state in brief some elementary facts concerning M\"obius geometry, mainly focusing on properties of the similarity group.

\subsection{The $\K$-Heisenberg group $\fH_\K$ and the Kor\'anyi metric $d_\fH$}\label{sec:hyp} 
The following are well known; we refer the reader to the classical book of Mostow, \cite{M}, or to \cite{P}. Another useful reference for the case $\K=\C$ is the book of Goldman, \cite{Gol}. For the octonionic case in particular, we refer to \cite{A} and to \cite{M-P}.

\medskip

We shall use hereafter the following notation ($n>1$):
$$
{\rm G}_\K=\left\{\begin{matrix} {\rm SO}(n,1)& & \text{if}\; \;\K=\R,\\
{\rm SU}(n,1)& & \text{if}\;\; \K=\C,\\
{\rm Sp}(n,1)& & \text{if}\; \;\K={\mathbb H},\\
F_{4(-20)} & & \text{if}\; \;\K={\mathbb O}\;(n=2).\end{matrix}\right.
$$
Also,
$$
{\rm F}(n)=\left\{\begin{matrix} {\rm SO}(n)& & \text{if}\; \;\K=\R,\\
{\rm SU}(n)& & \text{if}\;\; \K=\C,\\
{\rm Sp}(n)& & \text{if}\; \;\K={\mathbb H},\\
{\rm Spin}_7(\R) & & \text{if}\; \;\K={\mathbb O}\;(n=2).\end{matrix}\right.
$$
$\K$-hyperbolic space $\bH^n_\K$ is ${\rm G}_\K/{\rm F}(n)$. The  metric $d_\fH$ as given in Equation (\ref{eq:Kor}) is invariant under the following transformations (and their extensions to $\infty$).
\begin{enumerate}
 \item Left translations which come from the action of $\fH_\K$ on itself: For any fixed point $(\zeta',v')\in\fH_\K$ let
$$
T_{(\zeta',v')}(\zeta,v)=(\zeta',v')*(\zeta,v), \quad T_{(\zeta',v')}(\infty)=\infty.
$$
In the particular case where $\K=\R$, $d_\fH$  is also invariant under the right translations. 
\item Rotations: For the cases $\K\neq{\mathbb O}$, these come from the action of ${\rm F}(n-1)$ on $\K^{n-1}$. Specifically, given a $U\in{\rm F}(n-1)$, $\K\neq{\mathbb O}$, we define
$$
S_U(\zeta,v)=(U\cdot\zeta,v),\quad S_U(\infty)=\infty.
$$
Only in the case where $\K=\mH$ we have the action of ${\rm F}(1)={\rm Sp}(1)$ given by
$$
(\zeta_1,\dots,\zeta_{n-1},v)\mapsto (\mu\zeta_1\mu^{-1},\dots,\mu\zeta_{n-1}\mu^{-1},\mu v\mu^{-1}),\quad \mu\in{\rm Sp}(1);
$$
observe that in all other cases this action is vacuous.
In the particular case $\K={\mathbb O}$, for given unit imaginary octonion $\mu$, let
\begin{equation*}
S_{\mu}(\zeta,v)=(\zeta\cdot\overline{\mu},\mu v\overline{\mu}),\quad S_{\mu}(\infty)=\infty.
\end{equation*}
We stress at this point that in general $S_\mu\circ S_\nu\neq S_{\mu\nu}$ for $\mu,\nu$ unit imaginary octonions. The group generated by transformations $S_\mu$ is the compact group ${\rm Spin}_7(\R)$.
\end{enumerate}
All these actions generate the group ${\rm Isom}_{d_\fH}(\partial(\bH_\K^n)$ of (orientation-preserving) $d_\fH$-{\it isometries}; this acts transitively on $\fH_\K$. 
We also consider two other kinds of transformations of $\partial\bH_\K^n$.
\begin{enumerate}
\item[{(3)}] Dilations: If $\delta\in\R_*^+$ we define
$$
D_\delta(\zeta,v)=(\delta\zeta,\delta^2v),\quad D_\delta(\infty)=\infty.
$$
(In the octonionic case, $\delta^2$ is used in \cite{M-P} instead of $\delta$, but the model for $\fH_{\mathbb{O}}$ is somewhat different).
One verifies that for every $(\zeta,v),(\zeta',v')\in\partial\bH_\K^n$ we have
$$
d_\fH\left(D_\delta(\zeta,v),D_\delta(\zeta',v')\right)=\delta\cdot d_\fH\left((\zeta,v),(\zeta',v')\right).
$$
Thus the metric $d_\fH$ is  scaled up to multiplicative constants by the action of dilations. We mention here that together with $d_\fH$-isometries, dilations generate the $d_\fH$-{\it similarity group} ${\rm Sim}_{d_\fH}(\partial\bH_\K^n)$.
\item[{(4)}] Inversion $I$ is given by
$$
I(\zeta,v)=\left(\zeta(-\|\zeta\|^2+v)^{-1}\;,\;\overline{v}\left|-\|\zeta\|^2+v\right|^{-2}\right),\;\;\text{if}\;(\zeta,v)\neq o,\infty,\;\quad I(o)=\infty,\;I(\infty)=o.
$$
Inversion $I$ is an involution of $\partial\bH^n_\K$. Moreover, for all $p=(\zeta,v),p'=(\zeta',v')\in\fH_\K\setminus\{o\}$ we have
$$
d_\fH(I(p),o)=\frac{1}{d_\fH(p,o)},\quad d_\fH(I(p),I(p'))=\frac{d_\fH(p,p')}{d_\fH(p,o)\;d_\fH(o,p')}.
$$
\end{enumerate}

\medskip

The similarity group ${\rm Sim}_{d_\fH}(\partial\bH_\K^n)$ is the semidirect product $\R\rtimes{\rm Isom}_{d_\fH}(\partial\bH_\K^n)$. The group generated by ${\rm Sim}_{d_\fH}(\partial\bH_\K^n)$ and inversion $I$ is isomorphic to $G_\K$ (with the exception of the case $\K=\R$, $n$ odd: in this case, $I$ reverses orientation). 
Given two distinct points on the boundary, we can find an element of $G_\K$  mapping those points to $o$ and $\infty$ respectively; in particular $G_\K$  acts doubly transitively on the boundary. In the exceptional case where $\K=\R$, the action of $G_\R$  is triply transitive; this follows from the fact that we can map three distinct points of the boundary to the points $o$, $\infty$ and ${\bf e}_1$, respectively.

\subsection{M\"obius Group, Similarity Group}\label{sec:Mob}

Recall from the introduction that in the most general setting we start from a mere metric space $(S,d)$ and we suppose that there is a {\it remote point} $\infty$ such that:
\begin{enumerate}
 \item $\overline{S}=S\cup\{\infty\}$ is compact;
 \item $d$ is extended to the compactification $\overline{S}$ by setting
 $
 d(p,\infty)=+\infty$, 
for each $p\in S$ and also $d(\infty,\infty)=0$.
\end{enumerate}

The notion of metric cross-ratio is next. The following proposition may be verified straightforwardly, detailed discussions about cross-ratios may be found among others in \cite{F}, \cite{FP}, \cite{KR1} and \cite{P}:

\medskip

\begin{prop}\label{prop:abs-cr}
 Let $(S,d)$ be a metric space with a remote point $\infty$. Denote by $\cC$ the space $S\times S\times S\times S\setminus\{diagonals\}$. Then the map $|\X^d|:\cC\to \R_*^+$ defined for each $\fp=(p_1,p_2,p_3,p_4)\in\cC$ by
 \begin{equation}\label{eq:abs-cr}
  |\X^d|(\fp)=\frac{d(p_4,p_2)}{d(p_4,p_1)}\cdot\frac{d(p_3,p_1)}{d(p_3,p_2)},
 \end{equation}
satisfies the following:
\begin{enumerate}
 \item {\it Symmetries:}
 $$
 |\X^d|(p_1,p_2,p_3,p_4)=|\X^d|(p_2,p_1,p_4,p_3)=|\X^d|(p_3,p_4,p_1,p_2)=|\X^d|(p_4,p_3,p_2,p_1).
 $$
 \item Let
 $$
 |\X^d_1|(\fp)=|\X^d|(p_1,p_2,p_3,p_4),\quad\text{and}\quad |\X^d_2|(\fp)=|\X^d|(p_1,p_3,p_2,p_4).
 $$
 Then
 \begin{eqnarray*}
  &&
  |\X^d|(p_1,p_3,p_4,p_2)=\frac{1}{|\X^d_2|},\quad |\X^d|(p_1,p_4,p_3,p_2)=\frac{|\X^d_1|}{|\X^d_2|},\\
  &&
   |\X^d|(p_1,p_2,p_4,p_3)=\frac{1}{|\X^d_1|},\quad |\X^d|(p_1,p_4,p_2,p_3)=\frac{|\X^d_2|}{|\X^d_1|}.
 \end{eqnarray*}
Thus for any possible permutation $\tilde\fp$ of points of a given quadruple $\fp$, $|\X^d(\tilde\fp)|$ depends only on $|\X^d_1|(\fp)$ and $|\X^d_1|(\fp)$. 
\end{enumerate}

\end{prop}

A homeomorphism $\phi:S\to S$ shall be called a {\it $d$-M\"obius map} if $|\X^d(\fp)|=|\X^d(\phi(\fp))|$ for each $\fp=(p_1,p_2,p_3,p_4)\in\cC$, where, by $\phi(\fp)$ we denote the quadruple $\left(\phi(p_1),\phi(p_2),\phi(p_3),\phi(p_4)\right)$. This is equivalent to say that for each given quadruple $\fp$,
$$
|\X^d_1|(\phi(\fp))=|\X^d_1|(\fp) \quad\text{and} \quad |\X^d_2|(\phi(\fp))=|\X^d_2|(\fp).
$$
The set of all M\"obius maps of $S$ form a group $\cM_d=\cM_d(S)$, which we shall call the {\it M\"obius group} of $(S,d)$. 
In particular, we consider the subset ${\rm Sim}_d$ of $\cM_d$ consisting of similarities: An element $\phi\in\cM_d$ is a {\it similarity}, if there exists a $K_\phi>0$ such that for every $p,q\in S$
 $$
 d\left(\phi(p),\phi(q)\right)=K_\phi\cdot d(p,q).
 $$
 By definition, $K_\phi=K(\phi)$ is independent of the choice of points.
 
 \begin{prop}\label{prop:homoth}
The set ${\rm Sim}_d$ is a subgroup of $\cM_d$. There exists a group homomorphism $K:{\rm Sim}_d\to$ $ \R^+_*$ given for each $\phi\in{\rm Sim}_d$ by
\begin{equation*}\label{eq:homoth}
K(\phi)=K_\phi.
\end{equation*}
Here $\R^+_*$ is the multiplicative topological group of the positive reals, inheriting its topology from the real line.
\end{prop}
\begin{proof}
We show that ${\rm Sim}_d$ is a subgroup of $\cM_d$. If $\phi,\psi\in{\rm Sim}_d$, then for every $p,q\in S$,
 $$
 d\left(\phi(p),\phi(q)\right)=K_\phi\cdot d(p,q)\quad\text{and}\quad d\left(\psi(p),\psi(q)\right)=K_\psi\cdot d(p,q).
 $$
 Therefore for each $p,q\in S$,
 \begin{eqnarray*}
  d\left((\phi\circ\psi)(p),(\phi\circ\psi)(q)\right)&=&d\left((\phi\left(\psi(p)\right),(\phi\left(\psi
  (q)\right)\right)\\
  &=&K_\phi\cdot d\left(\psi(p),\psi(q)\right)\\
  &=&K_\phi\cdot K_\psi\cdot d(p,q).
 \end{eqnarray*}
Since this shows that $\phi\circ\psi$ is a similarity, we must also have
$$
d\left((\phi\circ\psi)(p),(\phi\circ\psi)(q)\right)=K_{\phi\circ\psi}\cdot d(p,q).
$$
Hence we also must have $K_{\phi^{-1}}=(K_\phi)^{-1}$.  
\end{proof}

We call $K$ the {\it similarity homomorphism}.  Under our assumptions for $(S,d)$ in the beginning of this section, we endow $\cM_d$ with the compact-open topology and accordingly ${\rm Sim}_d$ inherits the relative topology. We then have:

\begin{prop}\label{prop:cont}
The similarity homomorphism $K$ is continuous.
\end{prop}
\begin{proof}
 For any $\phi\in {\rm Sim}_d$ we have
 $$
 K(\phi)=\frac{d(\phi(p),\phi(q))}{d(p,q)},
 $$
 for any fixed $p,q\in S$. Now the map ${\rm Sim}_d\to S$, $\phi\mapsto\phi(p)$, is continuous in the compact-open topology for any fixed $p$, and $d:S\times S\to\R^+$ is continuous by definition. 
\end{proof}

We are interested in the stabiliser of $\infty$ in $\cM_d$, that is,
$$
{\rm Stab}_d(\infty):={\rm Stab}_{\cM_d}(\infty)=\{\phi\in\cM_d\;|\;\phi(\infty)=\infty \}.
$$

\begin{prop}\label{prop:K}
${\rm Stab}_d(\infty)={\rm Sim}_d$. 
\end{prop}

\begin{proof}
We show first that ${\rm Sim}_d\subseteq {\rm Stab}_d(\infty)$. Supposing the contrary, assume that there exists a $\phi\in {\rm Sim}_d$ such that $\phi(\infty)=p\in S$. Then for any $q\in S$, $q\neq p$, we would have
$$
+\infty=d\left(\phi^{-1}(p),\phi^{-1}(q)\right)=K_{\phi^{-1}}\cdot d(p,q),
$$
a contradiction.

To show that ${\rm Stab}_d(\infty)\subseteq {\rm Sim}_d$ we only have to prove that any element $\phi\in{\rm Stab}_d(\infty)$  is a similarity.
For this, we consider an arbitrary $\phi\in{\rm Stab}_d(\infty)$ and  we fix two arbitrary points $p,q\in S$. We also  consider the quantity
$$
K(\phi,p,q)=\frac{d\left(\phi(p),\phi(q)\right)}{d(p,q)}.
$$
Now for any $r\in S$, let $\fp=(r,q,\infty,p)$. Relations
$$
|\X^d_1|(\fp)=|\X^d_1|\left(\phi(\fp)\right)
\;\;\text{and}\;\;
|\X^d_2|(\fp)=|\X^d_2|\left(\phi(\fp)\right),
$$
yield
$$
K(\phi,p,q)=K(\phi,p,r)\;\;\text{and}\;\;K(\phi,p,q)=K(\phi,q,r).
$$
From the left equality it follows that $K(\phi,p,q)$ does not depend on $q$. But then the right equality shows that it also does not depend on $p$. Since $p,q$ are  arbitrary, $K=K(\phi)$ 
and the proof is complete.
\end{proof}

In view of Propositions \ref{prop:homoth} and \ref{prop:K} we obtain:

\begin{cor}\label{cor:K} If the similarity homomorphism $K$ is onto, then
\begin{equation*}
{\rm Sim}_d/{\rm Isom}_d\simeq \R^+_*, 
\end{equation*}
where ${\rm Isom}_d$ is the subgroup of $\cM_d$ consisting of isometries and $\R^+_*$ is the multiplicative group of positive real numbers. Thus we have the short exact sequence
\begin{equation}\label{eq:short-ex}
1\to {\rm Isom}_d\to {\rm Sim}_d\to \R^+_*\to 1.
\end{equation}
\end{cor}

\begin{rem}
Note that if there exists an inverse $L:\R^+_*\to{\rm Sim}_d$ to  the similarity homomorphism  $K$ of the short exact sequence \ref{eq:short-ex} such that $L\circ K$ is the identity homomorphism when restricted to ${\rm Isom}_d$, then by Splitting Lemma, see \cite{H}, we have 
$$
{\rm Sim}_d\simeq \R^+_*\rtimes{\rm Isom}_d,
$$ 
the semidirect product of $R^+_*$ and ${\rm Isom}_d$. This is for instance the case where $S=\partial\bH^n_\K$ and $d=d_\fH$: $L$ maps each positive $\delta$ to $D_\delta$ and $L\circ K=id.$ when restricted to ${\rm Isom}_{d_\fH}$. 
\end{rem}

We close this section with some remarks concerning inversions. In general, an {\it inversion} between two distinct  points $r,r'\in\overline{S}$ is a map $\phi_{r,r'}\in\cM_d$ such that $\phi^2=id.$ and $\phi(r)=r'$. Of course, an inversion may or may not exist in an arbitrary $\cM_d$. But if it does, it satisfies properties listed in the next proposition; for simplicity, we only treat the case where $r=\infty$ and $r'=o$ is an arbitrary point in $S$.

\begin{prop}\label{prop:invol}
 Let $\phi\in\cM_d$ be an inversion between $o$ and $\infty$. Then there exists a $\beta>0$ such that
 \begin{equation*}\label{eq:invol-cond1}
  d(o,p)\cdot d(o, \phi(p))=\beta^2, 
 \end{equation*}
for each $p\in S$ other than $o$, $\infty$. Moreover,
\begin{equation*}\label{eq:invol-cond2}
 d(\phi(p),\phi(q))=\frac{d(p,q)}{d(o,p)\cdot d(o,q)}\cdot \beta^2,
\end{equation*}
for each $p,q\in S$ other than $o$, $\infty$. 
\end{prop}
\begin{proof}
 We consider two arbitrary points $p,q\in S$ other than $o$ and $\infty$ and the quadruple $\fp=(\infty,o,$ $p,q)$. Since $\phi\in\cM_d$ we have
 $
 |\X^d_1|(\fp)=|\X^d_1|\left(\phi(\fp)\right),
 $
 which yields
 $$
 \frac{d(o,q)}{d(o,p)}=\frac{d(\phi(p),o)}{d(\phi(q),o)}.
 $$
 Thus the quantity $d(o,p)\cdot d(0,\phi(p))$ is constant for each $p$. By setting $p=p_0$ for some arbitrary $p_0$ other that $o$ and $\infty$ and by letting $\beta$ be the positive square root of $d(o,p_0)\cdot d(0,\phi(p_0))$, we obtain the first relation of our proposition.
 
On the other hand, from
 $
 |\X^d_2|(\fp)=|\X^d_2|\left(\phi(\fp)\right),
 $
 we have
$$
 \frac{d(p,q)}{d(o,p)}=\frac{d(\phi(p),\phi(q))}{d(o,\phi(q))}.
 $$
 Therefore,
 $$
 d(\phi(p),\phi(q))=d(o,\phi(q))\cdot\frac{d(p,q)}{d(o,p)}=\frac{d(p,q)}{d(o,p)\cdot d(o,q)}\cdot \beta^2.
 $$
 \end{proof}
 
  Note finally that from Proposition \ref{prop:invol} it follows that inversion $\phi$ leaves invariant the metric sphere $$S^\beta_d=\{p\in S\;|\; d(p,o)=\beta\}.$$

\section{M\"obius Rigidity}\label{sec:proof1}
In this section we prove our results. In Section \ref{sec:prep-lem} we prove an elementary lemma which will lead us to the proof of Theorems  \ref{thm:main1}, \ref{thm:main}, \ref{thm:mainw2} and \ref{thm:pt} which are in Sections \ref{sec:main1}, \ref{sec:main}, \ref{sec:mainw2} and \ref{sec:pt}, respectively.

\subsection{The Preparatory Lemma}\label{sec:prep-lem}
 The following holds.
\begin{lem}\label{lem:step1} We refer to the conditions stated in the introduction. Then conditions {\bf (Sim)} and {\bf (Top)} together imply the following.
\begin{enumerate} 
\item If $D_\delta$ is a dilation, then 
$K(\delta)=\delta^\alpha$, 
where $\alpha\in(0,1]$.
\item The rotational group $\rF(n-1)$ is in ${\rm Isom}_d(\partial\bH^n_\K)$.
\item The group of left translations of $\fH_\K$ is in ${\rm Isom}_d(\partial\bH^n_\K)$. 
 \end{enumerate}
\end{lem}
\begin{proof}
 To prove (1),  we will show that if $D_d$ is a dilation, then there exists an $\alpha\in(0,1]$ such that for each $p,q$,
 \begin{equation*}\label{eq:dil}
  d\left(D_\delta(p),D_\delta(q)\right)=\delta^\alpha\cdot d(p,q).
 \end{equation*}
 For this, set $L=\log K$. Then $L$ is a continuous function satisfying the functional equation
 $$
 L(\delta_1\cdot\delta_2)=L(\delta_1)+L(\delta_2),\quad \text{for all}\;\;\delta_1,\delta_2>0.
 $$
 Hence $L(\delta)=\alpha\log (\delta)$ and therefore
$K(\delta)=\delta^\alpha$ for some non-zero $\alpha$ as desired. 
Now $\alpha$ cannot be negative, for if that was the case then for every $\delta>0$ we would have
$
d\left(o,(\delta,0)\right)=\delta^\alpha\cdot d\left(o,(1,0)\right),$
where we have written  $({\bf e}_1,{\bf 0}_{\Im(\K)})=(1,0)$, $(\delta{\bf e}_1,{\bf 0}_{\Im(\K)})=(\delta,0)$. Since the function
$f(\delta)=$ $d\left(o,(\delta,0)\right)$ is continuous due to the continuity of $d$, by letting $\delta$ tending to $+\infty$ we have a contradiction.
The proof that $\alpha$ is in $(0,1]$ lies after the proof of (3).


To prove (2), we observe that from  condition ({\bf Top}) we have that the group $\rF(n-1)$ is compact in the compact-open topology arising from $d$.  Therefore its $K$-image in the multiplicative group $\R^+_*$ has to be compact, i.e., equal to 1. Thus the continuous $K$ is constant in $\rF(n-1)$ and thus equal to $1$.
 
 
 Finally, to prove (3), let $T$ be a translation and set $\lambda=K(T)>0$. We have to show that $\lambda=1$; asserting the contrary, suppose with no loss of generality that $\lambda<1$. For any $p\in\partial\bH^n_\K$, the sequence $T^{\nu}(p)$ tends to $\infty$ as $\nu\to +\infty$ in the $d_\fH$-topology. On the other hand,
 $$
 d(p,T^{\nu}(p))\le d(p,T(p))+\dots+d(T^{\nu-1}(p),T^{\nu}(p))\le(1+\dots+\lambda^{\nu-1}(p)).
 $$
 But this is bounded in the $d$-topology and this contradicts ({\bf Top}).
 
 We finally prove that $K(\delta)=\delta^\alpha$, $\alpha\in(0,1]$. For this, observe that for every two arbitrary $\delta_1, \delta_2>0$ we have
$$
d\left(o, D_{\delta_1+\delta_2}(1,0)\right)=(\delta_1+\delta_2)^\alpha\cdot d\left(o,(1,0)\right)
$$
and also by invariance of translations and triangle inequality,
\begin{eqnarray*}
 d\left(o, D_{\delta_1+\delta_2}(1,0)\right)&=&d\left(o, (\delta_1+\delta_2,0)\right)\\
 &=&d\left((-\delta_1,0),(\delta_2,0)\right)\\
 &\le&\delta_1^\alpha\cdot d\left((-1,0),o\right)+\delta_2^\alpha\cdot d\left((1,0),o\right)\\
 &=&\delta_1^\alpha\cdot d\left((1,0),o\right)+\delta_2^\alpha\cdot d\left((1,0),o\right)\\
 &=&\left(\delta_1^\alpha+\delta_2^\alpha\right)\cdot d\left(o, (1,0)\right).
\end{eqnarray*}
But also $ d\left(o, D_{\delta_1+\delta_2}(1,0)\right)=\left(\delta_1+\delta_2\right)^\alpha\cdot d\left(o, (1,0)\right)$; therefore, $(\delta_1+\delta_2)^\alpha\le \delta_1^\alpha+\delta_2^\alpha$. By putting $\delta_1=\delta_2=\delta$ we have
$2^\alpha\cdot \delta^\alpha\le 2\cdot\delta^\alpha$, i.e., $2^{1-\alpha}\ge 1$ which can happen only if $\alpha\in(0,1]$.
 This concludes the proof.
\end{proof}

\begin{rem}\label{rem:equid}
 (Notation convention) It is clear from the proof of Lemma \ref{lem:step1} that
 $$
  d\left(o,(\pm{\bf e}_i,{\bf 0}_{\Im(\K)})\right)=d\left(o,({\bf e}_1,{\bf 0}_{\Im(\K)})\right).
 $$
 where ${\bf e}_i$ is the $i$-coordinate vector of $\K^{n-1}$. All these quantities shall be hereafter denoted by $d\left((1,0),o\right)$. Also
 $$
 d\left(o,({\bf 0}_{\K^{n-1}},\pm{\bf f}_i)\right)=d\left(o,({\bf 0}_{\K^{n-1}},{\bf f}_1)\right).
 $$
 where ${\bf f}_i$ is the $i$-unit vector of $\Im(\K)$. This is is true because according to Proposition \ref{prop:K}, $J$ is in ${\rm Sim}_d$; in particular,  $K(J)=1$. All these quantities shall be denoted by $d\left((0,1),o\right)$. 
\end{rem}

\subsection{Proof of Theorem \ref{thm:main1}}\label{sec:main1} 
For the case $\K=\R$ (case (1) of the theorem), let $\zeta,\zeta'\in\R^{n-1}$. We have 
\begin{eqnarray*}
 d\left(\zeta,\zeta'\right)&=& d\left(\zeta-\zeta',0\right)\\
 &=& d\left(\|\zeta-\zeta'\|\cdot\mu,0\right),\quad \mu\in{\rm O}(n-1),\\
 &=& d(0,1)\cdot\|\zeta-\zeta'\|^\alpha=d(0,1)\cdot d_\fH^\alpha(\zeta,\zeta').
\end{eqnarray*}
Here, the penultimate equation follows from the transitive action of ${\rm SO}(n-1)$ on $S^{n-1}$: We may map any $r\in S^{n-1}$ to $1={\bf e}_1$ via an element of ${\rm SO}(n-1)$. 

For $\K\neq\R$ (case (2) of the theorem), let  $$\beta_1=\max\left\{d\left((1,0),o\right), d\left((0,1),o\right)\right\},\;\;
\beta_2=\min\left\{d\left((1,0),o\right), d\left((0,1),o\right)\right\}.$$
Since for each $p=(\zeta,v)\in\fH_\K$:
\begin{eqnarray}
 &&\label{eq:dil1}
 d\left((\zeta,0),o\right)=d\left((1,0),o\right)\cdot\|\zeta\|^\alpha,\\
 &&\label{eq:dil2}
 d\left((0,v),o\right)=d\left((0,1),o\right)\cdot |v|^{\alpha/2},
\end{eqnarray}
from \ref{eq:dil1} and \ref{eq:dil2} we obtain the inequalities:
\begin{eqnarray*}
 &&\label{ineq:dil1}
 \beta_2\cdot \|\zeta\|^\alpha\le d\left((\zeta,0),o\right)\le\beta_1\cdot \|\zeta\|^\alpha,\\
 &&\label{ineq:dil2}
 \beta_2\cdot |v|^{\alpha/2}\le d\left((0,v),o\right)\le\beta_1\cdot |v|^{\alpha/2}.
\end{eqnarray*}
Therefore,
\begin{equation*}
 \beta_2\cdot d^\alpha_\fH(o,p)\le \left(d^{4/\alpha}\left((\zeta,0),o\right)+d^{4/\alpha}\left((0,v),o\right)\right)^{\alpha/4}\le\beta_1\cdot d^\alpha_\fH(o,p).
\end{equation*}
By triangle inequality and $d$-invariance of  translations we have
$$
d(p,o)\le d\left((\zeta,0),o\right)+d\left((0,v),o\right).
$$
We apply H\"older's inequality  with exponent $4/\alpha$   to obtain
\begin{eqnarray*}
 d(p,o)&\le&2^{(4-\alpha)/4}\cdot\left(d^{4/\alpha}\left((\zeta,0),o\right)+d^{4/\alpha}\left((0,v),o\right)\right)^{\alpha/4}\\
 &\le&\beta_1\cdot 2^{(4-\alpha)/4}\cdot\left(d_\fH^4\left((\zeta,0),o\right)+d_\fH^4\left((0,v),o\right)\right)^{\alpha/4}\\
 &=&\beta_1\cdot 2^{(4-\alpha)/4}\cdot d^\alpha_\fH(p,o),
\end{eqnarray*}
and Theorem \ref{thm:main1} follows.\qed

We wish to address  at one important issue at this point. As we have underlined  in the introduction, condition {\bf ($\alpha$-H\"ol)} is rather inadequate to describe in full the nature of metrics which satisfy {\bf (Sim)} and {\bf (Top)}; there might exist metrics which satisfy {\bf ($\alpha$-H\"ol)} on the one hand but on the other, their nature might be entirely different from that of $d_\fH$. A concrete example which illustrates this matter is that  of  the Carnot-Carath\'eodory metric $d_{cc}$ (for details about $d_{cc}$, see for instance \cite{CDPT}); it suffices only to consider the case $\K=\C$, $n=2$. Certainly, since ${\rm Sim}_{d_{cc}}={\rm Sim}_{d_\fH}$, $d_{cc}$ automatically  satisfies  conditions {\bf (Sim)} and {\bf (Top)} of Theorem \ref{thm:main1} (but satisfies neither {\bf (Inv)} nor {\bf (G)}). Since $\alpha=1=\beta_1$,  we have from Theorem \ref{thm:main1} that for each $p$,
$$
d_{cc}(p,o)\le 2^{3/4}\cdot d_{\fH}(p,0).
$$
The reader is invited to compare this inequality to the well known estimate
 $$
 \pi^{-1/2}\cdot d_\fH(p,o)\le d_{cc}(p,o)\le \cdot d_\fH(p,o).
 $$

\subsection{Proof of Theorem \ref{thm:main}}\label{sec:main}
The following lemma is an immediate corollary of Proposition \ref{prop:invol}; it will be needed for the proof of Theorem \ref{thm:main}.

\begin{lem}\label{lem:confinv}
 If inversion  $I$ is in $\cM_d(\fH_\K)$, then
 \begin{equation*}\label{eq:inv-cond1}
  d(o,p)\cdot d(o, I(p))=d(o,p_0)\cdot d\left(o,I(p_0)\right),
 \end{equation*}
 where $p_0$ is any point other than $o,\infty$ and for each $p\in\fH_\K$ other than $o$, $\infty$. Moreover
\begin{equation*}\label{eq:inv-cond2}
 d(I(p),I(q))=\frac{d(p,q)}{d(o,p)\cdot d(o,q)}\cdot d(o,p_0)\cdot d\left(o,I(p_0)\right),
\end{equation*}
for each $p,q\in\fH_\K$ other than $o$, $\infty$. 
\end{lem}

We proceed now to the proof of Theorem \ref{thm:main}. Recall that $J$ is in ${\rm Sim}_d$ and $K(J)=1$. By choosing $p_0=(1,0)$ which we shall denote simply with 1, the formulae in the statement of the proposition also read as:
$$
d(o,p)\cdot d(o, I(p))=d^2(o,1),\;\; d(I(p),I(q))=\frac{d(p,q)}{d(o,p)}\cdot d(o,q)\cdot d^2(o,1).
$$
Let now $p=(\zeta,v)\in\fH_\K$. Then by setting $\cA(p)=-\|\zeta\|^2+v$ we have
\begin{eqnarray*}
 d\left(I(p),o)\right)&=&d\left(\left(\zeta\cdot\frac{\cA(\overline{p})}{|\cA(p)|^2},\frac{\overline{v}}{|\cA(p)|^2}\right),o\right)\\
 &=&d\left(D_{|\cA(p)|^{-1}}\left(\zeta\cdot\frac{\cA(\overline{p})}{|\cA(p)|},\overline{v}\right),o\right)\\
&=&|\cA(p)|^{-\alpha}\cdot d\left(\left(\frac{ \cA(p)}{|\cA(p)|}\cdot\overline{\zeta},v\right),o\right)\\
&=&|\cA(p)|^{-\alpha}\cdot d\left(\left(-\overline{\zeta},v\right),o\right)\\
&=&|\cA(p)|^{-\alpha}\cdot d\left(\left(-\zeta,-v\right),o\right)\\
&=&|\cA(p)|^{-\alpha}\cdot d(p^{-1},o)\\
&=&|\cA(p)|^{-\alpha}\cdot d\left(T_p (p^{-1}),o\right)\\
&=&|\cA(p)|^{-\alpha}\cdot d\left(o,p\right).
\end{eqnarray*}
This gives
$$
d\left(I(p),o)\right)=\frac{d^2(o,1)}{d(o,p)}=|\cA(p)|^{-\alpha}\cdot d\left(o,p\right),
$$
and by Lemma \ref{lem:confinv} we conclude
$$
d^2(o,p)=d^2(o,1)\cdot |\cA(p)|^{\alpha}=d^2(o,1)\cdot d_\fH^{2\alpha}(o,p),
$$
which proves Theorem \ref{thm:main}.\qed

Note that implicit in the proof lies the following formula for the inversion $I$: For each $p$,
$$
I(p)=\left(D_{|\cA(p)|^{-1}}\circ J\circ R_{-\frac{\cA(p)}{|\cA(p)|}}\circ J\right)\left(p^{-1}\right).
$$ 
This formula only dictates the manner that inversion $I$ is defined; it cannot be regarded as formula which expresses $I$ as a compound of similarities, due to the presence of $p^{-1}$.

\subsection{Proof of Theorem \ref{thm:mainw2}}\label{sec:mainw2}

Suppose now that {\bf (G)} holds, that is,
\begin{equation*}\label{eq:pyth1}
d(o,p)=\left(d^{4/\alpha}\left((\zeta,0),o\right)+d^{4/\alpha}\left((0,v),o\right)\right)^{\alpha/4}.
\end{equation*}
Then $\beta_2\cdot d^\alpha_\fH(o,p)\le d(o,p)\le \beta_1\cdot d^\alpha_\fH(o,p)$. 
From $d$-invariance and $d_\fH$-invariance of translations we have:
$$
\beta_2\cdot d_\fH^\alpha(p,q)\le d(p,q)\le \beta_1\cdot  d_\fH^\alpha(p,q), 
$$
therefore $d$ and $d_\fH^\alpha$ are metrically equivalent. But there exists also a $\beta\ge 1$ depending on $\beta_1,\beta_2$ 
 such that
$$
\frac{1}{\beta}\cdot d_\fH^\alpha(p,q)\le d(p,q)\le \beta\cdot d_\fH^\alpha(p,q).
$$
If  {\bf (Eq)} holds, then $\beta_1=\beta_2=\beta$ and  $d=\beta\cdot d_\fH^\alpha$. In this manner we have proved Theorem \ref{thm:mainw2}.
\qed

We remark here that condition {\bf (biLip)} is equivalent to that the identity map from $(\partial\bH^n_\K, d)$ to $(\partial\bH^n_\K, d_\fH)$ is bi-Lipschitz. 

\medskip

Finally, we prove:

\begin{prop}\label{prop:p-l}
Suppose that condition {\bf (Sim)} holds. Then conditions {\bf (G)} is equivalent to the following:
\begin{enumerate}
\item [{{\bf (P-L)}}] For $\alpha\in (0,1]$ and for each $p,q\in\fH$,
\begin{eqnarray*}
\left|p*q\right|^{4/\alpha}+ \left|p^{-1}*q\right|^{4/\alpha}+\left|p*q^{-1}\right|^{4/\alpha}+\left|p^{-1}*q^{-1}\right|^{4/\alpha}&=&\\
2\left(\left|\Pi_{\K^{n-1}}(p*q)\right|^{4/\alpha}+\left|\Pi_{\K^{n-1}}\left(p^{-1}*q\right)\right|^{4/\alpha}\right)&+&\\
\left|\Pi_{\Im(\K)}(p*q)\right|^{4/\alpha}+ \left|\Pi_{\Im(\K)}\left(p^{-1}*q\right)\right|^{4/\alpha}+\left|\Pi_{\Im(\K)}\left(p*q^{-1}\right)\right|^{4/\alpha}+\left|\Pi_{\Im(\K)}\left(p^{-1}*q^{-1}\right)\right|^{4/\alpha},
\end{eqnarray*}
where $\Pi_{\K^{n-1}}$ and $\Pi_{\Im(\K)}$ are projections of $\fH$ to $\K^{n-1}$ and $\Im(\K)$, respectively. 
\end{enumerate} 
\end{prop}
\begin{proof}
 To show that {\bf (P-L)} implies {\bf (G)}, just set $q=o$ and use the fact that $\left|p^{-1}\right|=|p|$, due to $d$-invariance of translations which follows from {\bf (Sim)} (cf. Lemma \ref{lem:step1}):
 $$
d\left(o,p^{-1}\right)=d\left(T_p(o),T_p\left(p^{-1}\right)\right)=d(o,p).
$$
To show that {\bf (G)} implies {\bf (P-L)}, just apply {\bf (G)} for the points $p*q$, $p^{-1}*q$, $p*q^{-1}$ and $p^{-1}*q^{-1}$.
\end{proof}

\subsection{Proof of Theorem \ref{thm:pt}}\label{sec:pt}

Assuming condition {\bf ($\alpha$-Met)} holds, we have for each $p,q,r,s$ points in $\partial\bH^n_\K$ that
\begin{eqnarray*}
d^\alpha_\fH(p,r)\cdot d^\alpha_\fH(q,s)&\le&\left(d_\fH(p,q)\cdot d_\fH(r,s)+d_\fH(p,s)\cdot d_\fH(r,q)\right)^{\alpha}\\
&\le& d^\alpha_\fH(p,q)\cdot d^\alpha_\fH(r,s)+d^\alpha_\fH(p,s)\cdot d^\alpha_\fH(r,q),
\end{eqnarray*}
where the first inequality follows from the Ptolemaean property of $d_\fH$ and the second inequality holds because $\alpha\le 1$. Thus $d$ satisfies {\bf (Ptol)}.

Assume now that  {\bf ($\alpha$-Met)} and {\bf (Circ)} hold. Pick any $p,q,r,s$ lying in the Ptolemaean circle $\sigma$ and suppose with no loss of generality that $p$ and $s$ separate $q$ and $r$. We may also normalise so that $q=o$ and $r=\infty$. Then
\begin{eqnarray*}
 d_\fH^\alpha(s,o)+ d_\fH^\alpha(p,0)= d_\fH^\alpha(p,s).
\end{eqnarray*}
Therefore,
$$
d_\fH(p,s)=\left(d_\fH^\alpha(s,o)+ d_\fH^\alpha(p,0)\right)^{1/\alpha}\ge d_\fH(s,o)+ d_\fH(p,0),
$$
since $1/\alpha\ge 1$. But this contradicts triangle inequality unless $\alpha=1$. Moreover, we have that then $\sigma$ is an $\R$-circle and the proof is complete. \qed


\begin{thebibliography}{ZZ99}

\bibitem{A} D.J.~Alcock;
{\sl Reflection groups on the octave hyperbolic plane}.
J. of Algebra {\bf 213} (1998), 467--498.

\bibitem{BS} S.~Buyalo \& V. Schroeder;
{\sl M\"obius structures and Ptolemy spaces: boundary
at inﬁnity of complex hyperbolic spaces}.
ArXiv:1012.1699v1 [math.MG]

\bibitem{BFW} S.M.~ Buckley, K.~Falk \& D.J.~ Wraith;
{\sl Ptolemaic inequality and CAT(0)}
Glasgow Math. J. {\bf 51} (2009) 301–314.

\bibitem{CDPT} L.~Capogna \& D.~Danielli \& S.D.~Pauls \& J.T.~Tyson;
 {\sl An introduction to the Heisenberg group and the sub--Riemannian isoperimetric problem}.
 Progress in Mathematics, {\bf 259}. Birkhäuser Verlag, Basel, 2007. 

\bibitem{F} E.~Falbel;
{\sl Geometric structures associated to triangulations as fixed point sets of involutions.}
Topol. and its Appl. {\bf 154} (2007), no. 6, 1041-1052. Corrected version in

www.math.jussieu.fr/$\sim$falbel 

\bibitem{FP} E.~Falbel \& I.D.~Platis;
{\sl The ${\rm PU}(2,1)$-configuration space of four points in $S^3$ and the cross-ratio variety}.
Math. Ann. {\bf 340} (2008), no. 4, 935--962.

\bibitem{Gol} W.~Goldman;
{\sl  Complex hyperbolic geometry}.
Oxford Mathematical Monographs. Oxford Science Publications. The Clarendon Press, Oxford University Press, New York, 1999. 


\bibitem{H} A.~Hatcher.
{\sl Algebraic Topology}.
Cambridge University Press, 2002.



\bibitem{KR1} A.~Kor\'anyi \& H.M.~Reimann;
{\sl The complex cross ratio on the Heisenberg group}.
Enseign. Math. (2) {\bf 33} (1987), no. 3-4, 291--300. 

\bibitem{M} G.D.~Mostow;
{\sl Strong rigidity of locally symmetric spaces}.
Ann. Math. stud. {\bf 78}, Princeton University Press, New Jersey, 1973.

\bibitem{M-P} S.~Markham \& J.R.~Parker;
{\sl J\o{}rgensen's inequality for metric spaces with applications to the octonions}.
Adv. Geom. {\bf 7} no. 1., (2007), 19--38.



\bibitem{P} I.D.~Platis;
{\sl Cross-ratios and the Ptolemaean inequality in boundaries of symmetric spaces of rank 1}.
Geometriae Dedicata, {\bf 169}, 187--208, 2014.


\bibitem{P-q} I.D.~Platis;
{\sl The ${\rm Psp}(2,1)$--configuration space of four points in $S^7$}.
In preparation.

\bibitem{S} I.J.~Schoenberg;
{\sl A remark on M.M. Day's characterization of inner--product spaces and a conjecture of L.M. Blumenthal}.
Proc. Amer. Math. Soc.  {\bf 3} (1952),  961--964. 



\end{thebibliography}
\end{document}